\documentclass[12pt]{amsart}

\usepackage{latexsym,amssymb,amsfonts,amsmath,mathrsfs,float,hyperref,mathtools}
\usepackage{xcolor}
\hypersetup{
    colorlinks,
    linkcolor={black},
    citecolor={black},
}

  \renewcommand{\Pr}{\mbox{\rm Pr}}
  \newcommand{\Exp}{{\mathbb{E}}}

  \newcommand{\C}{\mathbb{C}} 
  \newcommand{\D}{\mathbb{D}} 
  \newcommand{\N}{\mathbb{N}} 
  \newcommand{\Z}{\mathbb{Z}} 
  \newcommand{\F}{\mathbb{F}} 
  \newcommand{\pmset}[1]{\{-1,1\}^{#1}} 
  \newcommand{\bset}[1]{\{0,1\}^{#1}} 

  \DeclareMathOperator{\conv}{Conv}

  \DeclareMathOperator{\ord}{ord}
  \DeclareMathOperator{\ent}{\mathcal N}
  
  \newcommand{\st}{\,\mid\,} 

  \newcommand{\eps}{\varepsilon}

  \newcommand{\ceil}[1]{\lceil{#1}\rceil}

  \newcommand{\bfi}{\mathbf{i}}
   
  \newcommand{\beq}{\begin{equation}}
  \newcommand{\eeq}{\end{equation}}
  \newcommand{\beqn}{\begin{equation*}}
  \newcommand{\eeqn}{\end{equation*}}
  \newcommand{\beqr}{\begin{eqnarray}}
  \newcommand{\eeqr}{\end{eqnarray}}
  \newcommand{\beqrn}{\begin{eqnarray*}}
  \newcommand{\eeqrn}{\end{eqnarray*}}
  \newcommand{\bmline}{\begin{multline}}
  \newcommand{\emline}{\end{multline}}
  \newcommand{\bmlinen}{\begin{multline*}}
  \newcommand{\emlinen}{\end{multline*}}

  \theoremstyle{plain}
  \newtheorem{theorem}{Theorem}[section]
  \newtheorem{lemma}[theorem]{Lemma}
  \newtheorem{proposition}[theorem]{Proposition}

  \theoremstyle{definition}

  \theoremstyle{remark}
  \newtheorem{remark}[theorem]{Remark}
  
  \renewenvironment{proof}[1][]{
    	\begin{trivlist}
     	\item[\hspace{\labelsep}{\em\noindent Proof#1:\/}]}
     	{{\hfill$\Box$}
    	\end{trivlist}
  }
  
  \makeatletter
  \newtheorem*{rep@theorem}{\rep@title}
  \newcommand{\newreptheorem}[2]{%
  \newenvironment{rep#1}[1]{%
  \def\rep@title{#2 \ref{##1}}%
  \begin{rep@theorem}}%
  {\end{rep@theorem}}}
  \makeatother

  \newreptheorem{lemma}{Lemma}
  \newreptheorem{theorem}{Theorem}
  \newreptheorem{corollary}{Corollary}
  \newreptheorem{proposition}{Proposition}
  \newreptheorem{conjecture}{Conjecture}

\begin{document}
\title[]{High-entropy dual functions over finite fields and locally decodable codes}
\author{Jop Bri\"{e}t \and Farrokh Labib}
\address{CWI, Science Park 123, 1098 XG Amsterdam, The Netherlands}
\email{$\{$j.briet, labib$\}$@cwi.nl}
\thanks{The authors were supported by the  Gravitation grant  NETWORKS-024.002.003 from the Dutch Research Council (NWO)}

\maketitle


\begin{abstract}
We show that for infinitely many primes~$p$, there exist dual functions of order~$k$ over~$\F_p^n$ that cannot be approximated in $L_\infty$-distance by polynomial phase functions of degree~$k-1$.
This answers in the negative a natural finite-field analog of a problem of Frantzikinakis on $L_\infty$-approximations of dual functions over~$\N$ (a.k.a.\ multiple correlation sequences) by nilsequences.
\end{abstract}

\section{Introduction}

For $k\geq 2$, integer vector $\bfi  = (i_1,\dots,i_k)\in \Z_{\geq 0}^k$ and finite abelian group~$G$, the associated set of \emph{order-$k$ dual functions} is given by
\beqn\label{def:dual_function}
\Delta_{\bfi}
=
\big\{\phi:y\mapsto \Exp_{x\in G}f_1(x + i_1y)\cdots f_k(x+i_ky) \st
f_i:G\to \D\big\},
\eeqn
where~$\D$ denotes the complex unit disc.
For example, if $A\subseteq G$ is a subset, ${\bfi} = (0,1,2)$ and $f_i = 1_A$ for each~$i\in [3]$, then~$\phi(y)$ is the fraction of three-term arithmetic progressions in~$A$ with common difference~$y$.

For applications in additive combinatorics and higher-order Fourier analysis, it is desirable to understand to what extent dual functions can be approximated by simpler functions.
If~$k = 2$, it follows from the Fourier inversion formula that one has the simple decomposition in terms of the characters:
\beq\label{eq:linear}
\phi(y) 
=\sum_{\chi\in \widehat{G}} \alpha_\chi \chi\big((i_2 - i_1)y\big),
\eeq
where $\|\alpha\|_{\ell_1} \leq 1$.
Similar decompositions exist for higher-order dual functions thanks to deep ``inverse theorems'' for the Gowers uniformity norms.
Inverse theorems roughly show that if~$f$ has large~$U^k$-norm, then~$f$ correlates with a function $\psi:G\to \D$ akin to a polynomial of degree at most~$k-1$. 
Here the ``linear''~$\psi$ are precisely the characters.
What exactly the ``higher-order characters'' are depends on the group~$G$. 
For finite vector spaces~$\F_p^n$ with $p \geq k$, they are the \emph{polynomial phase functions} 
\beqn
\psi(x) =  e^{2\pi i P(x)/p},
\eeqn 
where~$P\in \F_p[x_1,\dots,x_n]$ is a polynomial  of degree at most~$k-1$~\cite{Tao:2010}.
When $p < k$, one has to consider the larger class of non-classical polynomials~\cite{Tao:2012b}.
For the cyclic group~$\Z_N$, they are the $(k-1)$-step nilsequences (of bounded complexity)~\cite{Green:2012}.
Combined with the Hahn-Banach theorem, these inverse theorems imply that the decomposition~\eqref{eq:linear} generalizes for larger~$k$ in terms of higher-order characters of degree at most~$k-1$ up-to small $L_1$-error~\cite{Gowers:2010}.
More precisely, in the finite-field setting, this amounts to the following:

\begin{proposition}\label{prop:L1approx}
Let~$p \geq k+1$ be a prime and let $G = \F_p^n$.
Then, for any $\eps > 0$ and ~$\bfi \in \Z_{\geq 0}^k$, there is an $M = M(\eps,k,p)>0$ such that any dual function $\phi \in \Delta_{\bfi}$ can be decomposed as
\beq\label{eq:phi_approx}
\phi = \sum_{i = 1}^r \alpha_i \psi_i + \tau,
\eeq
where
$\alpha_1,\dots,\alpha_r\in \C$ satisfy $|\alpha_1| + \cdots + |\alpha_r| \leq M$,
 $\psi_1,\dots,\psi_r$ are polynomial phases of degree at most~$k-1$ and
$\|\tau\|_{L_1} \leq \eps$.
\end{proposition}

While facts like this (in particular over~$\Z_N$) can be useful in higher-order Fourier analysis~\cite{Gowers:2010}, for other applications in additive combinatorics it is preferable to have more precise control over the error function~$\tau$ in~\eqref{eq:phi_approx}.
A natural finite-field analog of a question raised by Frantzikinakis in~\cite[Problem~1]{Frantzikinakis:2016b}  (see also~\cite{Altman:2019}) asks if this error function can be bounded \emph{everywhere}, that is, if Proposition~\ref{prop:L1approx} still holds with~$\|\tau\|_{L_\infty} \leq \eps$.
The apparent expectation of a positive answer to Frantzikinakis's question motivated conjectures on a poorly-understood probabilistic variant of Szemer\'edi's theorem on arithmetic progressions (cf.\ Section~\ref{sec:ldcs}).
Our main result, however, shows that in the finite-field setting, the answer is negative.

\begin{theorem}\label{thm:main}
For infinitely many primes~$p$, there is a $k = k(p)\in\N$ and an integer vector~$\bfi\in \Z_{\geq 0}^k$ such that~\eqref{eq:phi_approx} cannot hold with~$\|\tau\|_{L_\infty}\leq \eps$.
\end{theorem}

Special cases of Theorem~\ref{thm:main} show that for $k = 3$ and~$p = 2^t - 1$ a Mersenne prime, the decomposition~\eqref{eq:phi_approx} requires polynomial phases of degree at least~$t$ for fixed $\eps,M$ and~$\|\tau\|_{L_\infty} \leq \eps$.
The largest known Mersenne prime as of January 2018 has~$t = 77,232,917$~\cite{GIMPS}.

\subsection{Locally decodable codes and random Szemer\'edi} 
\label{sec:ldcs}

The examples behind Theorem~\ref{thm:main} originate from constructions of special types of error-correcting codes called \emph{locally decodable codes} (LDCs).
These codes have the property that any single encoded message symbol can be retrieved from a codeword with good probability by reading only a tiny number of codeword symbols, even if the codeword is partially corrupted.
LDCs originated in complexity theory~\cite{BK95, AS98,ALMSS98} and cryptography~~\cite{CKGS98} and were  defined  in the context of channel coding in~\cite{Katz:2000}.
They have since found many other applications in computer science and mathematics, for instance in fault tolerant distributed storage systems~\cite{GHSY12} and Banach space geometry~\cite{Briet:2012d}. 
We refer to~\cite{Yekhanin:2012, Gopi:2018} for extensive surveys.

Despite their ubiquity, LDCs are poorly understood.
Of particular interest is the tradeoff between the codeword length~$N$ as a function of message length~$k$ when the \emph{query complexity}---the number of probed codeword symbols---and alphabet size are constant.
The Hadamard code is a 2-query LDC of length~$N = 2^{O(k)}$ and this length is optimal in the 2-query regime~\cite{Kerenidis:2004}.
For~$q\geq 3$, the best-known lower bounds show that any $q$-query LDC has at least polynomial  length $k^{1+1/(\ceil{q/2}-1)-o(1)}$~\cite{Kerenidis:2004,Woodruff:2007a}.
The family of Reed-Muller codes, which generalize the Hadamard code, were for a long time the best-known examples, giving $q$-query LDCs of length~$ \exp(O(k^{1/(q-1)}))$.

In a breakthrough result, Yekhanin~\cite{Yekhanin:2007} constructed an entirely new family of vastly shorter  LDCs.
For each Mersenne prime $p = 2^t - 1$, he gave a 3-query LDC of length $N \leq \exp(O(k^{1/t}))$.
The construction uses a family of~$k$ homomorphisms from~$\F_p^n$ to the multiplicative subgroup of~$\F_{2^t}$.
The homomorphisms are constructed using a family of \emph{matching vectors} $(u_i,v_i)_{i\in[k]}$, which are pairs of orthogonal vectors in~$\F_p^n$ such that the inner products $\langle u_i,v_j\rangle$ with $i\ne j$ belong to a special subset of~$\F_p^*$.
It is this construction that forms the basis for Theorem~\ref{thm:main}.

Subsequently, Efremenko~\cite{Efremenko:2009} constructed much larger matching vector families over~$\Z_m^n$ for composite moduli~$m$ and used Yekhanin's framework to give the first 3-query LDCs of subexponential length $N \leq \exp(\exp(O\sqrt{\log k\log\log k}))$.
But huge gaps persist between the best-known upper and lower bounds for constant-query LDCs.\\

In contrast with other combinatorial objects such as expander graphs, the probabilistic method has so far not been successfully used to beat the best explicit LDC constructions.
In~\cite{BDG:2019}, a probabilistic framework was given that could in principle yield best-possible LDCs, albeit non-constructively.
A special instance of this framework connects LDCs with a probabilistic version of Szemer\'edi's theorem alluded to above. The setup for this is as follows:

For a finite abelian group~$G$ of size $N =|G|$, let $D\subseteq G$ be a random subset where each element is present with probability~$\rho$ independently of all others.
For $k \geq 3$ and $\eps\in (0,1)$, let~$E$ be the event that every subset $A\subseteq G$ of size $|A|\geq \varepsilon |G|$ contains a proper $k$-term arithmetic progression with common difference in~$D$. 
If ~$\rho = 1$, then it follows from the Density Hales--Jewett Theorem~\cite{FurstenbergK:1991} that~$E$ holds with probability~1 provided~$N$ is large enough in terms of~$k$ and~$\eps$.
It is an open problem to determine the smallest value of~$\rho$ --- which we will denote by~$\rho_k$ --- such that $\Pr[E] \geq \frac{1}{2}$.
This value will depend on~$\eps$ too, but we will suppress this in the notation and assume~$\eps$ is a fixed constant.
It is also assumed  that $N$ is large enough so that~$\rho_k$ exists.

In~\cite{BDG:2019} it is shown that there exist~$k$-query LDCs of message length~$\Omega(\rho_k N)$ and codeword length~$O(N)$.
As such, Szemer\'edi's theorem with random differences, in particular lower bounds on~$\rho_k$, can be used to show the existence of LDCs.
Conversely, this connection indirectly implies the  best-known upper bounds on~$\rho_k$ for all $k\geq 3$, given by $N^{-(1 - o(1))/\ceil{k/2}}$ \cite{Frantzikinakis:2012,BG:2018}. 
However, a conjecture of Frantzikinakis, Lesigne and Wierdl~\cite{FLW16} states that over $\Z_N$ we have $\rho_k \ll_kN^{-1}\log N$ for all~$k$, which would be best-possible.
Truth of this conjecture would imply that over this group,  Szemer\'edi's theorem with random differences cannot give LDCs better than the Hadamard code.
For finite fields, Altman~\cite{Altman:2019} showed that this conjecture is false.
In particular, over~$\F_p^n$ for~$p$ odd, he proved that~$\rho_3 \gg p^{-n}\, n^2$; generally, $\rho_k \gg p^{-n}\, n^{k-1}$  holds when $p \geq k+1$~\cite{Briet:2019}.
In turn, these bounds are  conjectured to be optimal for the finite-field setting, which would imply that over finite fields, Szemer\'edi's theorem with random differences cannot give LDCs better than Reed-Muller codes. 

These conjectures appear to be motivated mainly by the possibility of an $L_\infty$-version of Proposition~\ref{prop:L1approx} (and analogous variants over~$\Z_N$) with dual functions based on 3-term progressions.
Theorem~\ref{thm:main} falls short of obstructing this route to obtaining optimal bounds in the finite-field setting for two reasons.
First, our examples do not include  ``arithmetic-progression dual functions,'' those with ${\bf i} = (0,1,\dots,(k-1))$; in fact in the Appendix we show that our current framework cannot give such examples.
Second, even if we had such examples, they do not appear to imply any new lower bounds on~$\rho_k$.
Nevertheless, we do not expect arithmetic progressions to be exceptional patterns for which there are no such examples.

\begin{remark}
Ideas behind Theorem~\ref{thm:main} recently inspired similar examples in the integer setting for 3-term progressions in joint work of the first author and Green~\cite{BG:2020}.
\end{remark}

\subsection*{\bf Acknowledgements}
We thank Xuancheng Shao for introducing us to the problem of estimating entropy numbers of dual functions and for showing us Proposition~\ref{prop:L1approx},
Daniel Altman, Ben Green and Nikos Frantzikinakis for helpful discussions,
Igor Shparlinski for pointing us to the reference~\cite{stewart2013divisors},
and an anonymous referee for pointing out an error in a previous claim on a strengthening of the main result under the assumption of the generalized Riemann hypothesis.

\section{Preliminaries}

We will identify the set of maps $G\to\C$ with~$\C^G$.
For a polynomial $P(x)=\sum_{\iota =0}^tc_\iota x^\iota$, define its \emph{support} $\bfi(P)$ to be the sequence of degrees $\iota\in \Z_{\geq 0}$ such that $c_{\iota} \ne 0$, arranged in increasing order.
The \emph{support size} is the length of~$\bfi(P)$.
We will use some basic facts from the theory of finite fields, for which we refer to \cite{lidl1997finite}.
The Minkowski sum of two sets $A,B\subseteq \C^n$ is the set given by
\beqn
A+B = \{a+b\st a\in A,\, b\in B\}.
\eeqn
We will use the following slight generalization of the notion of the convex hull, where we allow for complex coefficients. For a compact set~$A\subseteq \C^n$, define
\beqn
\conv_\C(A) = \Big\{\sum_{a\in A} \alpha_a a \st \alpha_a\in \C \quad \forall a\in A, \quad\sum_{a\in A}|\alpha_a| \leq 1\Big\}.
\eeqn
For a finite set $A\subseteq \D^n$ and $\eps,M\in (0,\infty)$, define $\ent(A,\eps ,M)$ to be the smallest size of a finite set $B\subseteq M\D^n$ such that \beqn
A\subseteq \conv_\C(B) + \eps\D ^n.
\eeqn
Then, for any $a\in A$, there is a $b\in \conv_\C(B)$ such that $\|a - b\|_{\ell_\infty} \leq \eps$
and so $\ent(A,\eps,M)$ is a restricted form of the covering number of~$A$ relative to the~$\ell_\infty$ distance.
Note that for $I\subseteq [n]$, the projection of~$A$ to the set of coordinates~$I$, given by $A_I = \{(a_i)_{i\in I} \st a\in A\}$, is contained in $\conv_\C(B_I) + \eps\D^I$.
Since $|B| \geq |B_I|$, it follows that
\beq\label{eq:project}
\ent(A,\eps ,M) \geq \ent(A_I,\eps,M).
\eeq

\section{Covering numbers from hypercubes}

We will use the following lemma, which shows that containment of a high-dimensional hypercube implies a large restricted covering number.

\begin{lemma}\label{lem:cube}
Let $c>0$,  $z\in \C$ be a complex number such that $\Re(z) \leq 0$ and let $S\subseteq \C^k$ be a finite set such that $\{c,z\}^k \subseteq S$. 
Then, for any $\eps \in (0,\frac{c}{2})$ and $M > 0$, we have that 
\beqn
\log_2\big(\ent(S,\eps ,M)\big)
\gg_{c,\eps,M}
k.
\eeqn
\end{lemma}

\begin{proof}
Let~$\theta$ be a uniformly distributed $\pmset{k}$-valued random vector.
For a compact set $A\subseteq \C^k$, define
\beqn
w(A) = \Exp \max_{a\in A} |\langle a,\theta\rangle|.
\eeqn
We use the following basic properties:
\begin{enumerate}
\item[1.] If $A\subseteq B$, then $w(A) \leq w(B)$.
\item[2.] For a finite set $A\subseteq \C^k$, it holds that $w(\conv_\C(A)) = w(A)$.
\item[3.] For $A,B\subseteq \C^k$ finite, it holds that $w(A+B) \leq w(A) + w(B)$.
\end{enumerate}

It follows from the first property that 
\beq\label{eq:lbw}
w(S) \geq w(\{c,z\}^k) \geq \frac{ck}{2}.
\eeq
For the second inequality, observe that for fixed $\theta \in \pmset{k}$, we have
\begin{align*}
\max_{a\in \{c,z\}^k}|\langle a,\theta\rangle| 
&\geq
\Big|\sum_{i: \theta_i  = 1} c - \sum_{i: \theta_i = -1}z\Big|\\
&\geq 
\Big|\Re\Big(\sum_{i: \theta_i  = 1} c - \sum_{i: \theta_i = -1}z\Big)\Big|\\
&\geq   
c|\{i\in [k] \st \theta_i = 1\}|.
\end{align*}
Averaging over~$\theta$ then gives the result.

Let $B\subseteq M\D^k$ be a finite set such that $S\subseteq \conv_\C(B) + \eps\D^k$.
Let $l = |B|$ and $p = \log_2 l$.
By the second property of~$w$, Jensen's inequality and the Khintchine inequality~\cite[Chapter~5]{MilmanSchechtman},
\begin{align*}
w\big(\conv_\C(B)\big)
&=
\Exp \max_{b\in B}|\langle b,\theta\rangle|\\
&\leq
\Exp\Big( \sum_{b\in B}|\langle b,\theta\rangle|^p\Big)^{\frac{1}{p}}\\
&\leq
\Big( \sum_{b\in B}\Exp|\langle b,\theta\rangle|^p\Big)^{\frac{1}{p}}\\
&\ll
\sqrt{p}\Big( |B| \max\{\|b\|_{\ell_2}^p \st b\in B\}\Big)^{\frac{1}{p}}\\
&\ll
M \sqrt{k\log l}.
\end{align*}

We also have $w(\eps\D^k) = \eps k$.
Since $S \subseteq \conv_\C(B) + \eps\D^k$, the second and third properties of~$w$ and~\eqref{eq:lbw} then give
\beqn
\frac{ck}{2}
\leq
w(S)
\leq
w(\conv_\C(B)) + \eps\D^k) \ll M\sqrt{k\log_2 l} + \eps k .
\eeqn
Rearranging the left- and right-hand sides now gives the claim.
\end{proof}
\medskip

\section{Locating high-dimensional hypercubes}\label{section:yekhanin}

Here we show that for certain primes~$p$ and some integer vectors~${\bf i}$, the dual functions in~$\Delta_{\bf i}$ over~$\F_p^n$ contain high-dimensional hypercubes.

\begin{proposition}\label{prop:finite_field}
Let $p,r$ be distinct primes, let $t=\ord_p(r)$ and let $G = \F_p^n$.
Suppose there exists a polynomial $P(x)\in \F_{r}[x]$ that has a root in~$\F_{r^t}^*$ of order~$p$ and such that $P(1) \ne 0$. 
Then, there exists a $z\in \C$ with $\Re(z) \leq 0$ and a set $D\subseteq G$ of size $|D| \gg_p n^t$ such that
\beqn
\{z,1\}^D\subseteq \Delta_{\bfi(P)}^D.
\eeqn
\end{proposition}

The proof of this proposition relies on the following result due to Yekhanin, which is implicit in~\cite{Yekhanin:2007} (and shown explicitly in~\cite{Raghavendra:2007}).
We include a proof for completeness.

\begin{theorem}[Yekhanin]\label{thm:yekhanin}
Let $p,r$ be distinct primes and $t:=\ord_p(r)$. 
For integer~$m>p-1$, let
\beqn
k = {m\choose p-1}
\quad\quad
\text{and}
\quad\quad
n = {m + \frac{p-1}{t}- 1\choose \frac{p-1}{t}}.
\eeqn 
Let 
\beqn
P(x)=\sum_{\iota =0}^sc_\iota x^\iota\in \F_r[x]
\eeqn 
be  a polynomial with a root~$\gamma\in \F_{r^t}^*$ of order~$p$. 
Then, for each $i\in [k]$ there exists a function $f_i\colon \F_p^n\to\F_{r^t}$ and vectors $d_i,w_i\in\F_p^n$ such that for every $x\in \F_p^n$, we have
\beqn
\sum_{\iota =0}^sc_\iota f_i(x+\iota d_j)
=
\left\{\begin{array}{ll}
\gamma^{\langle x,w_i\rangle}P(1) & \text{if $i = j$}\\
0 & \text{otherwise.}
\end{array}
\right.
\eeqn 
\end{theorem}

\begin{proof}
For a $(p-1)$-element subset~$S\subseteq[m]$, define the vectors $u_S = 1_S$ and $v_S = 1_{[m]} - u_S$ in~$\F_p^m$.
Then, $\langle u_S,v_T\rangle =0$ if and only if $S= T$.
Let~$l = \frac{p-1}{t}$.
Then, for $a\in \F_p^*$, we have $a^l\in \{r^q\st q=0,1,\dots,p-1\}$.

Consider the expansion of the polynomial $Q(x)\in \F_p[x_1,\dots,x_m]$ given by
\beqn
Q(x) = (x_1 + \cdots + x_m)^l = 
\sum_{\beta\in \mathcal M_l} c_\beta x^\beta,
\eeqn
where $\mathcal M_l:=\{\beta\in \Z_{\geq 0}^m\st \sum_{i=1}^m\beta_i=l\}$ and $x^\beta:=\prod_{i=1}^mx_i^{\beta_i}$.
For each subset $S\subseteq[m]$ of size~$p-1$, define the vectors $w_S = (u_S^\beta)_{\beta\in \mathcal M_l}$ and $d_S = (c_\beta v_S^\beta)_{\beta\in \mathcal M_l}$.
Since $x^\beta y^\beta = (x\circ y)^\beta$, where $\circ$ denotes the coordinate-wise product, we have that
\beqn
\langle w_S, d_T\rangle = Q(u_S\circ v_T) = \langle u_S,v_T\rangle^l.
\eeqn
By the above, this equals zero if $S = T$ and a power of~$r$ otherwise.
Moreover, the vectors $w_S$ and~$d_S$ have dimension $|\mathcal M_l|  = {m+l-1\choose l}$.

Define $f_S:\F_p^n \to \F_{r^t}^*$ by
\beqn
f_S(x) = \gamma^{\langle x, w_S\rangle}.
\eeqn
Note that this a homomorphism, because $\gamma$ has order $p$. Then,
\begin{align*}
\sum_{\iota =0}^sc_\iota f_S(x+\iota d_S)
&=
\gamma^{\langle x,w_S\rangle}\sum_{\iota=0}^sc_\iota \gamma^{\iota \langle d_S,w_S\rangle}\\
&=
\gamma^{\langle x,w_S\rangle}\sum_{\iota=0}^sc_\iota\\
&=\gamma^{\langle x,w_S\rangle}P(1).
\end{align*}
If $S\ne T$, then $\langle d_T,w_S\rangle = r^q\bmod{p}$ for some integer~$q$ and therefore,
\begin{align*}
\sum_{\iota =0}^sc_\iota f_S(x+\iota d_T)
&=
\gamma^{\langle x,w_S\rangle}\sum_{\iota=0}^sc_\iota \gamma^{\iota \langle d_T,w_S\rangle}\\
&=
\gamma^{\langle x,w_S\rangle}\sum_{\iota =0}^sc_\iota\gamma^{\iota r^q}\\
&=
\gamma^{\langle x,w_S\rangle}P(\gamma)^{r^q}\\
&=0.
\end{align*}
This completes the proof.
\end{proof}

\medskip

\begin{proof}[ of Proposition~\ref{prop:finite_field}]
Let~$P(x) \in \F_r[x]$ be as in Proposition~\ref{prop:finite_field} and let~$\gamma\in \F_{r^t}^*$ be a $p$-th root of unity  such that~$P(\gamma) = 0$.
Let $f_i:\F_p^n\to \F_{r^t}^*$ and $d_i,w_i\in \F_p^n$ be as in Theorem~\ref{thm:yekhanin}. 
Let $\chi:\F_{r^t}\to \C$ be a nontrivial additive character such that the complex number
\beqn
z:= \Exp_{c\in \F_p}\chi\big(\gamma^cP(1)\big)
\eeqn
satisfies $\Re(z) \leq 0$.
To see that such a character exists, observe that by orthogonality of the characters,
\beqn
\Exp_{\chi \in \widehat{\F_{r^t}}} \Exp_{c\in \F_p}\chi\big(\gamma^cP(1)\big) = 
\Exp_{c\in \F_p}\Big(
\Exp_{\chi \in \widehat{\F_{r^t}}}
\chi\big(\gamma^cP(1)\big)
\Big)=
0.
\eeqn
The existence of the desired character then follows by averaging.
For each $a\in \bset{k}$ and $\iota\in \bfi(P)$, define $F^\iota_a:\F_p^n\to \C$ by
\beq\label{eq:dualF}
F^\iota_a(x) = \chi\Big(c_\iota\sum_{j=1}^k a_j f_j(x)\Big).
\eeq
Based on these functions, we define the dual function $\phi_a:\F_p^n \to \D$ by
\beq\label{eq:dualphi}
\phi_a(y) = \Exp_{x\in \F_p^n}\prod_{\iota\in \bfi(P)} F_a^{\iota}(x + \iota y).
\eeq
Then,
\begin{align*}
\phi_a(d_i)
&=
\Exp_{x\in \F_p^n} \chi\Big(\sum_{j=1}^k a_j \sum_{\iota\in \bfi(P)}c_\iota f_j(x+\iota d_i)\Big)\\
&=
\Exp_{x\in \F_p^n}\chi(a_i \gamma^{\langle x,w_i\rangle}P(1))\\
&=
\Exp_{c\in \F_p}\chi(a_i \gamma^cP(1)).
\end{align*}
The last expectation equals~1 if $a_i = 0$ and~$z$ if~$a_i = 1$ and therefore, 
\beqn
\{1,z\}^k \subseteq \{(\phi(d))_{d\in D} \st \phi\in \Delta_{\bfi(P)}\}.
\eeqn 
Since $k \geq (\frac{m}{p})^{p-1}$, $n \leq (\frac{2etm}{p})^{\frac{p-1}{t}}$ and $t \leq p-1$, we have $k \gg_p n^t$.
\end{proof}

\section{Sparse polynomials over $\F_2$}\label{section:sparse}

The following lemma supplies infinitely many primes and  polynomials that can be used in Proposition~\ref{prop:finite_field} .

\begin{lemma}\label{lemma:inf_many_sparse_pols}
For infinitely many primes~$p$, there is an irreducible polynomial $P(x)\in \F_2[x]$ with support size at most~$t = \ord_p(2)$ and  a root in $\F_{2^t}^*$ of order~$p$.
\end{lemma}

To prove Lemma~\ref{lemma:inf_many_sparse_pols}, we use some basic theory of cyclotomic polynomials (see for example~\cite[Chapter 2]{lidl1997finite}).
Let $r$ be a prime and $n\in\N$ not divisible by $r$. 
Recall that a primitive $n$-th root of unity over~$\F_r$ is a generator of the non-zero elements of the splitting field of the polynomial $x^n-1$ over $\F_r$.
Then, for any such root of unity~$\zeta$, the $n$-th cyclotomic polynomial is given by
\begin{align*}
\Phi_n(x)=\prod_{\gcd(s,n)=1}(x-\zeta^s),
\end{align*}  
where the product is over $s\in\{1,\dots,n\}$ such that $\gcd(s,n)=1$. 
The following lemma gives the properties of cyclotomic polynomials we need.

\begin{lemma}\label{lem:cyclo}
Let~$r$ be a prime,~$n\in \N$ not divisible by~$r$.
Then, the coefficients of $\Phi_n(x)$ lie in $\F_r$.
Moreover, if~$n$ is a prime, then~$\Phi_n(x)$ factors into $(n-1)/\ord_n(r)$ distinct monic irreducible polynomials all of which have degree exactly~$\ord_n(r)$.
\end{lemma}

For an integer $k\geq 2$, denote by $p(k)$ the largest prime number that divides $k$. We will use the following result of Stewart~\cite{stewart2013divisors}. 

\begin{lemma}[Stewart]\label{lemma:order_of_2_modp}
	For all $n$ large enough, we have
	\begin{align*}
	p(2^n-1)>n\exp\left(\frac{\log n}{104\log\log n}\right).
	\end{align*}
\end{lemma}

\begin{proof}[ of Lemma \ref{lemma:inf_many_sparse_pols}]
By Lemma~\ref{lemma:order_of_2_modp}, for  $p=p(2^n-1)$ and~$n$ sufficiently large, we have $\ord_p(2)\leq n < (p-1)/2$.
Hence, there are infinitely many primes $p$ such that $t:=\ord_p(2)\leq(p-1)/2$. 
For such a~$p$,  consider the $p$-th cyclotomic polynomial $\Phi_p(x)$ over~$\F_2$. 
By Lemma~\ref{lem:cyclo}, $\Phi_p(x)$ factors into $(p-1)/t$ distinct monic irreducible polynomials over $\F_2$ of degree exactly~$t$. 
Since over $\F_2$, there is only one polynomial of degree $t$ with support size $t+1$, there must be an irreducible factor with support of size at most $t$. 
Let~$P(x)$ be such a factor. 
Then, since $P(x)|\Phi_p(x)$,  its roots lie in the set of $p$-th roots of unity in~$\F_{2^t}$. 
\end{proof}

\begin{remark}
For Mersenne primes $p=2^t-1$, there are polynomials over~$\F_2$ with support size~3 that meet the conditions of Proposition~\ref{prop:finite_field}.
Indeed, since in this case, any $p$-th root of unity~$\zeta$ in~$\F_{2^t}$ is a generator of~$\F_{2^t}^*$ and since~$1 + \zeta \ne 0$, there exists an $s$ such that $P(x) = 1 + x + x^s$ satisfies $P(1) =1$ and $P(\zeta) = 0$.
\end{remark}

\section{Proof of Theorem~\ref{thm:main}}
\label{sec:outline}

Let $p,t,P(x)$ be as in Lemma~\ref{lemma:inf_many_sparse_pols}, so that~$P$ has support size~$k \leq t$.
Let ${\bf i} = {\bf i}(P)$.
Since~$P$ is irreducible,~$P(1)\ne 0$ and so it satisfies the conditions of Proposition~\ref{prop:finite_field}.
Fix $\eps \in (0,\frac{1}{2})$ and $M\in (0,\infty)$.
Suppose that Proposition~\ref{prop:L1approx} held with $\|\tau\|_{L_\infty} \leq \eps$, which is to say that
\beqn
\Delta_{\bf i}
\subseteq
\conv_\C\big(M\cdot \{\text{polynomial phases of degree $\leq k-1$}\}\big) + \eps\D^{\F_p^n}.
\eeqn
Then, since there are at most $p^{O(n^{k-1})}$ polynomial phase functions of degree at most~$k-1$ (one for each $n$-variate polynomial of degree at most~$k-1$), this implies that
\beq\label{eq:lowent}
\log_2\ent(\Delta_{\bfi},\eps ,M)
\ll_p n^{k-1}
\ll_p n^{t-1}.
\eeq
At the same time, Proposition~\ref{prop:finite_field}, Lemma~\ref{lem:cube} and property~\eqref{eq:project} give
\beqn
\log_2\ent(\Delta_{\bf i},\eps,M) \gg_{p,\varepsilon,M} n^t.
\eeqn
This contradicts~\eqref{eq:lowent} for large~$n$ and finishes the proof of Theorem~\ref{thm:main}.

\appendix
\section{On the possible arithmetic patterns}

Here we show that our construction cannot give examples for dual functions corresponding to arithmetic progressions.
Let $p,r$ be primes and~$t = \ord_p(r)$.
Suppose that for some $k,s\in \N$, there is a polynomial $P(x)\in \F_r[x]$ of the form
\beqn
P(x) = \sum_{\iota=0}^{k-1}c_\iota x^{\iota s}
\eeqn
such that $P(1) \ne 0$ and $P(x)$ has a root in~$\F_{r^t}^*$ of order~$p$.
Then, the functions defined as in~\eqref{eq:dualF} and~\eqref{eq:dualphi} belong to the set of dual functions corresponding to the progression $\bfi = (0,s,2s,\dots,(k-1)s)$ and generate in a hypercube of dimension at least~$n^t$.
We show that $k \geq t+1$, which means that this does not contradict an $L_\infty$-version of Proposition~\ref{prop:L1approx}.

First note that~$s$ cannot be a multiple of~$p$, since for any $\gamma\in \F_{r^t}^*$ of order~$p$ we would have $\gamma^{s} = 1$, which implies that $P(\gamma) = P(1) \ne 0$.
It follows that for any such~$\gamma$, the element~$\gamma^s$ also has order~$p$ and does not equal~1.
Define the polynomial 
\beqn
Q(x) = \sum_{\iota=0}^{k-1}c_\iota x^\iota \in \F_r[x].
\eeqn
Then, this polynomial has a root~$\alpha$ in~$\F^*_{r^t}$ of order~$p$ (where $\alpha = \gamma^s$), satisfies~$Q(1) = P(1) \ne 0$ and has degree~$k-1$.
We claim that $k-1 \geq \ord_p(r)$.
If~$Q$ is reducible, then it has a factor of degree strictly less than~$k-1$ that has the same properties.
So assume that~$Q$ is irreducible.
Let $K = \F_r(\alpha)$ be the simple algebraic extension of~$\F_r$ obtained by adjoining~$\alpha$.
Then~$K$ is isomorphic to~$\F_{r^{k-1}}$.
Since~$\alpha$ lies in $\F_{r^{k-1}}$ and has order~$p$,  it follows that $p\mid r^{k-1} - 1$.
But this implies that $k-1 \geq \ord_p(r) = t$.

\bibliographystyle{alphaabbrv}
\bibliography{dual_entropy}
\end{document}